\documentclass[twocolumn]{article}
\usepackage{amsmath}
\usepackage{amsthm}
\usepackage{amsfonts}
\usepackage{amssymb}
\usepackage{mathrsfs}
\usepackage{amscd}
\usepackage{url}

\newcommand{\re}{\mathbb{R}}

\newcommand{\N}{\mathbb{N}}

\newcommand{\lmd}{\lambda}

\newcommand{\sig}{\sigma}
\newcommand{\Sig}{\Sigma}

\newcommand{\reff}[1]{(\ref{#1})}

\newcommand{\prm}{\prime}

\newcommand{\bdes}{\begin{description}}
\newcommand{\edes}{\end{description}}
\newcommand{\bal}{\begin{align}}
\newcommand{\eal}{\end{align}}
\newcommand{\bnum}{\begin{enumerate}}
\newcommand{\enum}{\end{enumerate}}
\newcommand{\bit}{\begin{itemize}}
\newcommand{\eit}{\end{itemize}}
\newcommand{\bea}{\begin{eqnarray}}
\newcommand{\eea}{\end{eqnarray}}
\newcommand{\be}{\begin{equation}}
\newcommand{\ee}{\end{equation}}
\newcommand{\baray}{\begin{array}}
\newcommand{\earay}{\end{array}}
\newcommand{\bsry}{\begin{subarray}}
\newcommand{\esry}{\end{subarray}}
\newcommand{\bca}{\begin{cases}}
\newcommand{\eca}{\end{cases}}
\newcommand{\bcen}{\begin{center}}
\newcommand{\ecen}{\end{center}}
\newcommand{\bbm}{\begin{bmatrix}}
\newcommand{\ebm}{\end{bmatrix}}
\newcommand{\bmx}{\begin{matrix}}
\newcommand{\emx}{\end{matrix}}
\newcommand{\bpm}{\begin{pmatrix}}
\newcommand{\epm}{\end{pmatrix}}
\newcommand{\btab}{\begin{tabular}}
\newcommand{\etab}{\end{tabular}}

\newtheorem{theorem}{Theorem}

\theoremstyle{definition}

\begin{document}

\title{Local Versus Global Conditions in Polynomial Optimization}

\author{Jiawang Nie \\
Department of Mathematics, \\
  University of California San Diego,  \\
  La Jolla, CA 92093, USA \\
Email: njw@math.ucsd.edu
}

\date{}

\maketitle

This paper briefly reviews the relationship between
local and global optimality conditions in \cite{Nie-opcd}.
Consider the polynomial optimization problem
\be  \label{pop:gen}
\left\{\baray{rl}
\min & f(x) \\
s.t. & h_i(x) = 0 \,(i=1,\ldots,m_1),  \\
 & g_j(x) \geq  0 \,(j=1,\ldots,m_2),
\earay \right.
\ee
where $f, h_1,\ldots, h_{m_1}, g_1, \ldots, g_{m_2}$
are real polynomials in $x := (x_1, \ldots, x_n)$.
For convenience, denote
\[
h:=(h_1,\ldots,h_{m_1}), \quad
g:=(g_1,\ldots,g_{m_2})
\]
and $g_0:=1$.
Let $K$ be the feasible set of \reff{pop:gen}.
When there are no equality (resp., inequality) constraints,
the tuple $h = \emptyset$ and $m_1 =0$
(resp., $g = \emptyset$ and $m_2 =0$).

The problem \reff{pop:gen} can be treated as a general nonlinear program.
By classical nonlinear optimization methods,
we can typically get a
Karush-Kuhn-Tucker (KKT) point of \reff{pop:gen}.
Theoretically, it is NP-hard to check whether a KKT point
is a local minimizer or not. However,
it is often not too hard to do that in practice.
This is because there exist standard conditions
ensuring local optimality.
On the other hand, it is often much harder to get a global minimizer.
In practice, sometimes we may be able to get a global optimizer,
but it is typically hard to verify the global optimality.
A major reason for this is lack of easily checkable
global optimality conditions in nonlinear programming theory.

Local and global optimality conditions
are presumably very different, except special cases like convex optimization.
For general nonconvex optimization, little is known about
global conditions. However, for polynomial optimization,
this is possible by using representations of nonnegative polynomials.
Interestingly, global optimality conditions are
closely related to the local ones,
which was discovered in the paper \cite{Nie-opcd}.

\section{Local Optimality Conditions}

Let $u$ be a local minimizer of \reff{pop:gen} and
\[
J(u) \, := \, \{j_1, \ldots, j_r\}
\]
be the index set of active inequality constraints.
If the {\it constraint qualification condition (CQC)} holds at $u$,
i.e., the gradient vectors
\[
\nabla h_1(u), \ldots, \nabla h_{m_1}(u),
\nabla g_{m_1}(u), \ldots, \nabla g_{j_r}(u)
\]
are linearly independent, then
there exist Lagrange multipliers $\lmd_1,\ldots,\lmd_{m_1}$
and $\mu_{1},\ldots,\mu_{m_2} $ satisfying
\be  \label{opcd:grad}
\nabla f(u) =  \sum_{i=1}^{m_1} \lmd_i \nabla h_i(u) +
\sum_{j =1 }^{m_2} \mu_j \nabla g_j (u),
\ee
\be \label{op-cd:comc}
\left. \baray{c}
\mu_{1} g_{1}(u) = \cdots = \mu_{m_2} g_{m_2}(u) = 0,  \\
\mu_{1} \geq 0, \ldots,  \mu_{m_2} \geq 0.
\earay \right \}
\ee
The equation \reff{opcd:grad} is called the {\it first order optimality condition (FOOC)},
and \reff{op-cd:comc} is called the {\it complementarity condition}.
If it further holds that
\be \label{optcd-scc}
\mu_{1} + g_{1}(u) >0, \ldots,  \mu_{m_2} + g_{m_2}(u) > 0,
\ee
then the {\it strict complementarity condition (SCC)} holds at $u$.
The strict complementarity is equivalent to $\mu_j >0$ for every $j \in J(u)$.
Let $L(x)$ be the Lagrange function
\[
L(x) := f(x) - \sum_{i=1}^{m_1} \lmd_i h_i(x) -
\sum_{j\in J(u)} \mu_j  g_j (x).
\]
Clearly, \reff{opcd:grad} implies the gradient $\nabla_x L(u) = 0$.
The polynomials $f,h_i,g_j$ are smooth functions.
Thus, under the constraint qualification condition,
the {\it second order necessity condition (SONC)} holds:
\be \label{opt:sonc}
v^T \nabla_x^2 L(u) v \geq 0 \qquad  \forall \, v \in G(u)^\perp.
\ee
In the above, $G(u)$ denotes the Jacobian of the active constraining polynomials
\[
G(u) = \mbox{Jacobian} \Big(
h_1, \ldots,  h_{m_1}, g_{j_1}, \ldots, g_{j_r}
\Big)\Big|_{x=u}
\]
and $G(u)^\perp$ denotes the null space of $G(u)$. If it holds that
\be \label{sosc-optcnd}
v^T \nabla_x^2 L(u) v > 0 \quad  \mbox{ for all } \, 0 \ne v \in G(u)^\perp,
\ee
then the {\it second order sufficiency condition (SOSC)} holds at $u$.
The relations among the above conditions can be summarized as follows:
if CQC holds at $u$, then \reff{opcd:grad},
\reff{op-cd:comc} and \reff{opt:sonc} are necessary conditions for
$u$ to be a local minimizer, but they may not be sufficient;
if \reff{opcd:grad}, \reff{op-cd:comc}, \reff{optcd-scc} and \reff{sosc-optcnd}
hold at $u \in K$, then $u$ is a strict local minimizer of \reff{pop:gen}.
We refer to \cite[Section~3.3]{Brks} for such classical results.

Mathematically, CQC, SCC and SOSC are sufficient
for local optimality, but may not be necessary.
However, for {\it generic} cases, they are
sufficient and necessary conditions.
This is a major conclusion of \cite{Nie-opcd}.
Denote by $\re[x]_d$ the set of real polynomials in $x$ and
with degrees at most $d$. Let $[m]:=\{1,\ldots,m\}$.
The following theorem is from \cite{Nie-opcd}.

\begin{theorem} \label{mthm:opcd=gen}
Let $d_0, d_1, \ldots, d_{m_1}, d_1^{\prm}, \ldots, d_{m_2}^{\prm}$ be positive integers.
Then there exist a finite set of nonzero polynomials $\varphi_1,\ldots,\varphi_L$,
which are in the coefficients of polynomials
$f \in \re[x]_{d_0}$, $h_i \in \re[x]_{d_i}$ ($i\in [m_1]$),
$g_j \in \re[x]_{d_j^{\prm}}$ ($j\in [m_2]$) such that if
\[
\baray{c}
\varphi_1( f, h_1, \ldots,  h_{m_1}, g_{1}, \ldots, g_{m_2} ) \ne 0, \\
\vdots \\
\varphi_L( f, h_1, \ldots,  h_{m_1}, g_{1}, \ldots, g_{m_2} ) \ne 0,
\earay
\]
then CQC, SCC and SOSC
hold at every local minimizer of \reff{pop:gen}.
\end{theorem}

Theorem~\ref{mthm:opcd=gen} implies that
the local conditions CQC, SCC and SOSC hold at every local minimizer
in the space of input polynomials with given degrees,
except a union of finitely many hypersurfaces.
So, they hold in an open dense set in the space of input polynomials.
Therefore, CQC, SCC and SOSC can be used
as sufficient and necessary conditions
in checking local optimality,
for {\it almost all} polynomial optimization problems.
This fact was observed in nonlinear programming.

\section{A global optimality condition}

Let $u$ be a feasible point for \reff{pop:gen}.
By the definition, $u$ is a global minimizer if and only if
\be \label{df:gopt}
f(x) - f(u) \geq 0  \quad \forall \, x \in K.
\ee
Typically, it is quite difficult to check \reff{df:gopt} directly.
In practice, people are interested in easily checkable
conditions ensuring \reff{df:gopt}. For polynomial optimization,
this is possible by using sum-of-squares type representations.

Let $\re[x]$ be the ring of real polynomials in $x:=(x_1,\ldots,x_n)$.
A polynomial $p \in \re[x]$ is said to be {\it sum-of-squares} (SOS)
if $p=p_1^2+\cdots+p_k^2$ for $p_1,\ldots, p_k \in \re[x]$.
A sufficient condition for \reff{df:gopt}
is that there exist polynomials
$\phi_1, \ldots, \phi_{m_1} \in \re[x]$ and SOS polynomials
$\sig_0, \sig_1, \ldots, \sig_{m_2} \in \re[x]$ such that
\be \label{cond:gopt}
f(x) - f(u)  =  \sum_{i=1}^{m_1} \phi_i (x) h_i(x)
+  \sum_{j=0}^{m_2} \sig_j(x) g_j(x).
\ee
The equality in \reff{cond:gopt} is a polynomial identity in the variables of $x$.
Note that for every feasible point $x$ in \reff{pop:gen},
the right hand side in \reff{cond:gopt} is always nonnegative.
This is why \reff{cond:gopt} ensures that
$u$ is a global minimizer.
The condition \reff{cond:gopt} was investigated by
Lasserre~\cite{Las01}. It was a major tool for
solving the optimization problem \reff{pop:gen} globally.
We call \reff{cond:gopt} a global optimality condition
for \reff{pop:gen}.

People wonder when the global optimality condition holds.
The representation of $f(x)-f(u)$ in \reff{cond:gopt}
was motivated by Putinar's Positivstellensatz \cite{Put},
which gives SOS type certificates for positive or nonnegative polynomials
on the set $K$. Denote
\[
\langle h \rangle :=   h_1 \re[x] + \cdots +  h_{m_1} \re[x],
\]
which is the ideal generated by the polynomial tuple $h$.
Let $\Sig[x]$ be the set of all SOS polynomials in $\re[x]$.
The polynomial tuple $g$ generates the quadratic module:
\[
Q(g) := \Sig[x] +  g_1 \Sig[x] + \cdots +  g_{m_2} \Sig[x].
\]
If there exists a polynomial $p\in \langle h \rangle + Q(g)$
such that the set $\{x \in \re^n:\, p(x) \geq 0 \}$ is compact,
then $\langle h \rangle + Q(g)$ is said to be {\it archimedean}.
The archimedeanness of $\langle h \rangle + Q(g)$ implies the compactness
of $K$, while the reverse is not necessary. However, when $K$ is compact,
we can always add a redundant condition like $R- \|x\|_2^2 \geq 0$
to the tuple $g$ so that $\langle h \rangle + Q(g)$ is archimedean.
Hence, archimedeanness of $\langle h \rangle + Q(g)$ is almost equivalent to
the compactness of $K$. Putinar's Positivstellensatz \cite{Put} says that if
$\langle h \rangle + Q(g)$ is archimedean,
then every polynomial which is strictly positive on $K$
belongs to $\langle h \rangle + Q(g)$ (cf.~\cite{Put}).

The global optimality condition \reff{cond:gopt} 
is equivalent to the membership
\[
f(x) - f(u) \in \langle h \rangle + Q(g).
\]
When $u$ is a global minimizer of \reff{pop:gen}, the polynomial
\[
\tilde{f}(x):=f(x) - f(u)
\]
is nonnegative on $K$, but not strictly positive on $K$.
This is because $u$ is always a zero point of $\tilde{f}$ on $K$.
So, Putinar's Positivstellensatz itself does not imply
the global optimality condition \reff{cond:gopt}.
Indeed, there are counterexamples that
\reff{cond:gopt} may not hold. For instance,
when $f$ is the Motzkin polynomial
$x_1^2x_2^2(x_1^2+x_2^2-3x_3^2)+x_3^6$
and $K$ is the unit ball,
then \reff{cond:gopt} fails to hold.

However, the global optimality condition \reff{cond:gopt}
holds for {\it almost all} polynomials $f, h_i, g_j$, i.e.,
it holds in an open dense set in the space of input polynomials.
This is a major conclusion of \cite{Nie-opcd}.
The ideal $\langle h \rangle$ is said to be {\it real} if
every polynomial in $\re[x]$ vanishing on the set $\{x \in \re^n: \, h(x) = 0 \}$
belongs to $\langle h \rangle$ (cf.~\cite{BCR}). This is a general condition.
For instance, if $\langle h \rangle$ is a prime ideal
and $h$ has a nonsingular real zero,
then $\langle h \rangle$ is real (cf.~\cite{BCR}).
As pointed out earlier, when the feasible set $K$ is compact,
we can generally assume that $\langle h \rangle + Q(g)$ is archimedean.
Interestingly, the local conditions
CQC, SCC and SOSC imply the global optimality condition
\reff{cond:gopt}, under the archimedeanness of $\langle h \rangle + Q(g)$.
The following theorem
is a consequence of the results in \cite{Nie-opcd}.

\begin{theorem} \label{thm:gopt-cond}
Assume that the ideal $\langle h \rangle$ is real
and the set $\langle h \rangle + Q(g)$ is archimedean.
If the constraint qualification condition,
strict complementarity condition, and
second order sufficiency condition hold at every
global minimizer of \reff{pop:gen},
then the global optimality condition \reff{cond:gopt} holds.
\end{theorem}
\begin{proof}
At every global minimizer $u$ of $f$ on $K$,
the CQC, SCC and SOSC conditions implies that
the boundary hessian condition holds at $u$,
by Theroem~3.1 of \cite{Nie-opcd}.
The boundary hessian condition was introduced by Marshall \cite{Mar09}
(see Condition~2.3 of \cite{Nie-opcd}).
Let $f_{min}$ be the global minimum value of \reff{pop:gen}.
Denote $V = \{ x \in \re^n: \, h(x) = 0 \}$.
Let $I(V)$ be the set of all polynomials vanishing on $V$.
By Theorem~9.5.3 of \cite{MarBk}
(also see Theorem~2.4 of \cite{Nie-opcd}), we have
\[
f(x) - f_{min} \in I(V) + Q(g).
\]
Because $\langle h \rangle $ is real,
$\langle h \rangle = I(V)$ and
\[
f(x) - f(u) \in \langle h \rangle + Q(g).
\]
So, the global optimality condition \reff{cond:gopt} holds.
\end{proof}

By Theorem~\ref{mthm:opcd=gen}, the local
conditions CQC, SCC and SOSC hold generically, i.e.,
in an open dense set in the space of input polynomials.
Therefore, the global optimality condition \reff{cond:gopt}
also holds generically, when
$\langle h \rangle$ is real
and $\langle h \rangle + Q(g)$ is archimedean.

\section{Lasserre's hierarchy}

Lasserre \cite{Las01} introduced a sequence of semidefinite relaxations
for solving \reff{pop:gen} globally,
which is now called {\it Lasserre's hierarchy} in the literature.
It can be desribed in two equivalent versions.
One version uses SOS type representations, while the other one uses
moment and localizing matrices.
They are dual to each other, as shown in \cite{Las01}.
For convenience of description, we present the SOS version here.
For each $k \in \N$ (the set of nonnegative integers),
denote the sets of polynomials (note $g_0 = 1$)
\[
\langle h \rangle_{2k}: =
\left\{
\left. \overset{m_1}{ \underset{ i=1}{\sum} }  \phi_i h_i \right|
\baray{c}
\mbox{ each } \phi_i \in \re[x] \\
\mbox{ and } \deg (\phi_i h_i ) \leq 2k
\earay
\right\},
\]
\[
Q_k(g):= \left\{
\left. \overset{m_2}{ \underset{ j=0}{\sum} }   \sig_j g_j \right|
\baray{c}
\mbox{each } \sig_j \in  \Sig[x]  \\
\mbox{ and } \deg( \sig_j g_j) \leq 2k
\earay
\right\}.
\]
Note that $\langle h \rangle_{2k}$ is a truncation of $\langle h \rangle$
and $Q_k(g)$ is a truncation of $Q(g)$.
The SOS version of Lasserre's hierarchy is
the sequence of relaxations
\be  \label{sos:Put}
\max \quad \gamma \quad \mbox{s.t.} \quad
f - \gamma  \in  \langle h \rangle_{2k} + Q_k(g)
\ee
for $k=1,2,\ldots,$. The problem~\reff{sos:Put}
is equivalent to a {\it semidefinite program} (SDP).
So it can be solved as an SDP by numerical methods.
For instance, the software {\tt GloptiPoly~3}  \cite{GloPol3}
and {\tt SeDuMi} \cite{sedumi} can be used to solve it.
We refer to \cite{LasBok,Lau,MarBk} for recent work
in polynomial optimization.

Let $f_{min}$ be the minimum value of \reff{pop:gen} and
$f_k$ denote the optimal value of \reff{sos:Put}.
It was shown that (cf.~\cite{Las01})
\[
\cdots \leq f_k \leq  f_{k+1} \leq \cdots \leq f_{min}.
\]
When $\langle h \rangle + Q(g)$ is archimedean,
Lasserre \cite{Las01} proved the asymptotic convergence
\[
f_k \to f_{min} \quad \mbox{ as } \quad k\to \infty.
\]
If $f_k = f_{min}$ for some $k$,
Lasserre's hierarchy is said to have {\it finite convergence}.
It is possible that the sequence $\{f_k\}$ has only asymptotic, but not finite, convergence.
For instance, this is the case when $f$ is the Motzkin polynomial
$x_1^2x_2^2(x_1^2+x_2^2-3x_3^2)+x_3^6$
and $K$ is the unit ball \cite[Example~5.3]{Nie-jac}.
Indeed, such $f$ always exists whenever $\dim(K)\geq 3$,
which can be implied by \cite[Prop.~6.1]{Sch99}.
However, such cases do not happen very much.
Lasserre's hierarchy often has finite convergence in practice,
which was demonstrated by extensive numerical experiments
in polynomial optimization (cf.~\cite{HenLas03,HenLas05}).

A major conclusion of \cite{Nie-opcd} is that Lasserre's hierarchy
almost always has finite convergence.
Specifically, it was shown that
Lasserre's hierarchy has finite convergence
when the local conditions CQC, SCC and SOSC are satisfied,
under the archimedeanness.
The following theorem is shown in \cite{Nie-opcd}.

\begin{theorem} \label{mthm:opc=>fcvg}
Assume that $\langle h \rangle + Q(g)$ is archimedean.
If the constraint qualification, strict complementarity and
second order sufficiency conditions
hold at every global minimizer of \reff{pop:gen},
then Lasserre's hierarchy of \reff{sos:Put} has finite convergence.
\end{theorem}

By Theorem~\ref{mthm:opcd=gen}, the local
conditions CQC, SCC and SOSC at every local minimizer,
in an open dense set in the space of input polynomials.
This implies that, under the archimedeanness of
$\langle h \rangle + Q(g)$,
Lasserre's hierarchy has finite convergence,
in an open dense set
in the space of input polynomials.
That is, Lasserre's hierarchies
almost always (i.e., generically) have finite convergence.
This is a major conclusion of \cite{Nie-opcd}.

If one of the assumptions in Theorem~\ref{mthm:opc=>fcvg}
does not hold, then $\{f_k\}$ may fail to have finite
convergence. The counterexamples were shown in \S3 of \cite{Nie-opcd}.
On the other hand, there exists other non-generic
conditions than ensures finite convergence of $\{ f_k \}$.
For instance, if $h$ has finitely many real or complex zeros,
then $\{f_k\}$ has finite convergence (cf.~\cite{Lau07,Nie-PopVar}).

Since the minimum value $f_{min}$ is typically not known,
a practical concern is how to check
$f_k = f_{min}$ in computation.
This issue was addressed in \cite{Nie-ft}.
Flat truncation is generally a sufficient and
necessary condition for checking finite convergence.

For non-generic polynomial optimization problems,
it is possible that the sequence $\{f_k\}$
does not have finite convergence to $f_{min}$.
People are interested in methods that have finite convergence
for minimizing {\it all} polynomials over a given set $K$.
The Jacobian SDP relaxation proposed in \cite{Nie-jac}
can be applied for this purpose.
It gives a sequence of lower bounds that have finite converge
to $f_{min}$, for every polynomial $f$
that has a global minimizer over a general set $K$.

\bigskip
\noindent
{\bf Acknowledgement}
The research was partially supported by the NSF grants
DMS-0844775 and DMS-1417985.

\end{document}